\documentclass[12pt,reqno]{amsart}
\usepackage{}
\usepackage{bbm}
\usepackage{url}
\usepackage{stmaryrd}
\usepackage{mathrsfs}
\usepackage{cases}
\usepackage{amsfonts}
\usepackage{amssymb}
\usepackage{amsmath}
\usepackage{tikz}
\usepackage{extarrows}
\usepackage{enumerate}

\allowdisplaybreaks[4]
\newtheorem{theorem}{Theorem}[section]
\newtheorem{lemma}[theorem]{Lemma}
\newtheorem{proposition}[theorem]{Proposition}
\newtheorem{corollary}[theorem]{Corollary}
\theoremstyle{definition}
\newtheorem{definition}[theorem]{Definition}

\newtheorem{remark}[theorem]{Remark}
\numberwithin{equation}{section}

\let\d=\delta

\let\r=\rho

\let\om=\omega

\let\ep=\epsilon
\let\va=\varphi

\let\fy=\infty

\newcommand{\be}{\begin{equation*}}
\newcommand{\ee}{\end{equation*}}
\newcommand{\ben}{\begin{equation}}
\newcommand{\een}{\end{equation}}
\newcommand{\bn}{\begin{enumerate}}
\newcommand{\en}{\end{enumerate}}

\newcommand{\ba}{\begin{align}}
\newcommand{\ea}{\end{align}}

\def\calF{\mathcal {F}}

\def\rr{{\mathbb R}}
\def\rn{{{\rr}^n}}

\def\lpw{L^p_{\om}(\rn)}

\begin{document}
\title[On relatively compact sets in quasi-Banach function spaces]
{On relatively compact sets in quasi-Banach function spaces}
\author{WEICHAO GUO}
\address{School of Science, Jimei University, Xiamen, 361021, P.R.China}
\email{weichaoguomath@gmail.com}
\author{GUOPING ZHAO}
\address{School of Applied Mathematics, Xiamen University of Technology, Xiamen, 361024, P.R.China}
\email{guopingzhaomath@gmail.com}
\thanks{}

\begin{abstract}
This paper is devoted to the study of the relatively compact sets in Quasi-Banach function spaces,
providing an important improvement of the known results.
As an application, we take the final step in establishing a relative compactness criteria for function spaces with any weight without any assumption.
\end{abstract}
\thanks{Supported by the National Natural Foundation of China (Nos. 11701112, 11601456, 11671414, 11771388)
and  Natural Science Foundation of Fujian Province (Nos. 2017J01723, 2018J01430).}

\maketitle

\section{Introduction}\label{s1}

The characterization of relatively compactness in the classical $L^p$ Lebesgue spaces
was discovered by Kolmogorov (see \cite{Kolmogoroff1931NvdGdWzGMK,Tikhomirov1991}) under some restrictive conditions. Then
it was extended by Tamarkin \cite{Tamarkin1932BAMS} and Tulajkow \cite{Tulajkov}.
At the same time, M. Riesz \cite{Riesz1933ASSM} proved a similar result.
More precisely, the complete version of classical Riesz-Kolmogorov theorem can be stated as follows:
\\
\\
{\bf{Theorem A.}} (Classical Riesz-Kolmogorov theorem)
Let $1\leq p<\fy$. A subset $\calF$ of $L^p(\rn)$ is relatively compactness, if and only if the following three conditions hold:
\bn[(a)]
  \item $\calF$ is bounded, i.e., $\sup\limits_{f\in \calF}\|f\|_{L^p(\rn)}\lesssim 1$;
  \item $\calF$ uniformly vanishes at infinity, that is,
  \be
  \lim_{N\rightarrow \infty}\sup_{f\in \calF}\|f\chi_{B^c(0,N)}\|_{L^p(\rn)}= 0;
  \ee
  \item $\calF$ is equicontinuous, that is,
  \be
  \lim_{r\rightarrow 0}\sup_{f\in \calF}\sup_{y\in B(0,r)}\|\tau_yf-f\|_{L^p(\rn)}=0.
  \ee
\en
Here, $\tau_y$ denotes the translation operator: $\tau_yf(x)=f(x-y)$.

From then on, the compactness criterias were studied by many authors in various settings,
e.g. \cite{BandaliyevGorka2018BJMA,CaetanoGogatishviliOpic2016PRSESA,Gorka2014JAA,GorkaKostrzewa2016JAA,
GorkaMacios2015JFA,GorkaPospiechtoappearAoFA,Kalamajska1999AASFM,Pego1985PAMS,Rafeiro2009PARMI,Sudakov1957UMNN,Weil1940}.
Meanwhile,
it has played an important role in the compactness results of certain bounded operators in the field of harmonic analysis,
e.g. \cite{Uchiyama1978TMJ2,ChenDingWang2012CJMa, BenyiDamianMoenTorres2015MMJ, ChaffeeChenHanTorresWard2018PAMS, ClopCruz2013AASFM, WuYang2018PAMS}.
Among numerous of articles, we would like to mention some of them from the following two perspectives:
\bn
\item
{\bf{Extension to general settings.}}
First, the Lebesgue metric measure space $(\rn, |\cdot|, m)$ in the classical case, with Euclidean metric $|\cdot|$ and Lebesgue measure $m$,
can be generalized to the metric measure space $(X,\r, \mu)$ with metric $\r$ and measure $\mu$.
More precisely, one can study the relatively compactness property on $L^p(X,\r, \mu)$,
see \cite{GorkaMacios2014MMN, GorkaRafeiro2017AASFM, Krotov2012MS} for this direction.
In this general case, if one also wants to establish an equivalent characterization theorem on $L^p(X, \r, \mu)$ like Theorem A,
the condition (c) in Theorem A should be replaced by the following condition:
\be
(c^{+})
\ \ \ \ \ \ \ \ \ \
\lim_{r\rightarrow 0}\sup_{f\in \calF}\left\|\frac{1}{\mu(B(\cdot,r))}\int_{B(\cdot,r)}fd\mu-f\right\|_{L^p(X,\r,\mu)}=0,
\ee
where $B(x,r)=\{y\in X: \r(x,y)<r\}$.
Recently, in a more general framework, the compactness criteria were studied by G\'{o}rka--Rafeiro \cite{GorkaRafeiro2016NA} in the setting of Banach function space $E$ associated with $(X,\r, \mu)$.
The main result \cite[Theorem 3.1]{GorkaRafeiro2016NA} is a new relatively compact criteria fitting more general cases, in which
the equicontinuous condition is replaced by
\be
(c^{*})
\ \ \ \ \ \ \ \ \ \
\lim_{r\rightarrow 0}\sup_{f\in \calF}\left\|\frac{1}{\mu(B(\cdot,r))}\int_{B(\cdot,r)}fd\mu-f\right\|_{E}=0,
\ee
where $E$ is a Banach function space containing certain $\mu$-measurable functions.
In \cite{GorkaRafeiro2016NA}, the authors also establish a necessity result \cite[Thoerem 3.2]{GorkaRafeiro2016NA} under some reasonable assumptions of the Banach function spaces.

If $(X,\r, \mu)$ is restricted to the Lebesgue metric meaure space, Caetano-Gogatishvili-pic \cite{CaetanoGogatishviliOpic2016PRSESA}
establish a relatively compact criteria with condition (c), see Theorems 4.3 and Theorem 4.9 in \cite{CaetanoGogatishviliOpic2016PRSESA},
in which the converse is also established under the assumption of rearrangement invariant.

We would like to point out that in the case of Lebesgue metric measure space, (c*) can be deduced by (c).
In fact, by the Minkowski-type inequality (see Appendix A), we deduce that
\ben\label{intro, 1}
\begin{split}
&\left\|\frac{1}{|B(\cdot,r)|}\int_{B(\cdot,r)}f(y)dy-f\right\|_{E}
=
\left\|\frac{1}{|B(0,r)|}\int_{B(0,r)}f(\cdot+y)dy-f\right\|_{E}
\\
&\quad \quad=
\left\|\frac{1}{|B(0,r)|}\int_{B(0,r)}(\tau_{-y}f-f)dy\right\|_{E}
\\
&\quad \quad \leq
\frac{1}{|B(0,r)|}\int_{B(0,r)}\left\|\tau_{-y}f-f\right\|_{E}dy
\leq \sup_{y\in B(0,r)}\|\tau_yf-f\|_{E},
\end{split}
\een
where we assume that $h(y):=\|\tau_{-y}f-f\|_{E}$ is a measurable function.
While this inclusion relations between (c*) and (c) is invalid for general $\mu$ and $\r$.

\item
{\bf{Applications to Harmonic analysis.}}
In the field of harmonic analysis, in order to verify the compactness of a bounded operator, we usually apply the compactness criteria like Theorem A, for instance one can see \cite{Uchiyama1978TMJ2,ChenDingWang2012CJMa} for the unweighted case of compact commutator of singular integral,
and see \cite{BenyiDamianMoenTorres2015MMJ, ChaffeeChenHanTorresWard2018PAMS, ClopCruz2013AASFM, WuYang2018PAMS} for the weighted case.
Especially, in order to verify  the compactness of a $L^p_{\om}(\rn)-L^p_{\om}(\rn)$ bounded operator with some weight function $\om$,
the reasonable equicontinuous condition should be as follows:
\ben\label{i, 1}
\lim_{r\rightarrow 0}\sup_{f\in \calF}\sup_{y\in B(0,r)}\|\tau_yf-f\|_{L^p_{\om}(\rn)},
\een
one can see
\cite{Uchiyama1978TMJ2,ChenDingWang2012CJMa, BenyiDamianMoenTorres2015MMJ, ChaffeeChenHanTorresWard2018PAMS, ClopCruz2013AASFM, WuYang2018PAMS} for more details.
Note that the known results with condition $(c^{+})$ in the setting of metric measure space are invalid here, since condition \eqref{i, 1} can not imply
condition $(c^+)$ even when $\r=|\cdot|$ and $d\mu=\om dx$.
Even to this day, due to the incompleteness of weighted compactness criteria,
the weighted version of Riesz-Kolmogorov theorem is still being improved, one can see a very recent article \cite{XueYabutaYan2018apa},
in which the authors study the relatively compactness criteria for $\lpw$.
We remark that in \cite{XueYabutaYan2018apa}, some additional assumptions are still needed for the weights (see \cite[Theorem 4.2]{XueYabutaYan2018apa}), although
the assumptions permit the weights beyond the $A_{p}$ $(1\leq p<\fy)$ class.
\en
\bigskip
Based on the above two directions of research, we have two natural considerations:
\bn
\item Can the weighted version of Riesz-Kolmogorov theorem be deduced by a more general theorem established in \cite{CaetanoGogatishviliOpic2016PRSESA,GorkaRafeiro2016NA}
on (Quasi-)Banach function spaces?
\item Can the additional assumptions on weights be completely eliminated in the weighted version of Riesz-Kolmogorov theorem?
\en

The main purpose of this article is to consider the two problems mentioned above.
In fact, for general weights without additional assumption, the answer for the first problem is negative,
one will see the detailed explanation in Section 2.
In order to solve the second problem,
we turn to establish a useful relatively compactness criteria
in a suitable framework of (Quasi-)Banach function spaces, which is not included in \cite{CaetanoGogatishviliOpic2016PRSESA, GorkaRafeiro2016NA}.
As an application, we take the final step in establishing a relative compactness criteria for function spaces with any weight.

The remainder of this paper is structured as follows.
In Section 2, we give some required definitions and notations for the framework we are working on.
And, we also explain why the \cite{CaetanoGogatishviliOpic2016PRSESA,GorkaRafeiro2016NA} are invalid in our case.
Section 3 is devoted to the proofs of our main results, including a relatively compactness criteria on Banach and Quasi-Banach function spaces.
We also give an important application on weighted Lebesgue space.

We point out that in the setting of completed metric space, relatively compactness and totally boundedness are equivalent.

\section{Banach function spaces and weighted function spaces}\label{s2}
We first recall some basic definitions of function spaces.
In this paper, we only consider the class of Lebesgue measurable functions, denoted by $L(m)$, where
 $m$ means the Lebesgue measure on $\rn$.
\begin{definition}
A (quasi-)normed space $(E, \|\cdot\|_E)$ with $E\subset L(m)$ is called a (Quasi-)Banach function space ((Q-)BFS) if it satisfies the following conditions:
 \begin{description}\label{df, B}
   \item[$(B0)$] if $\|f\|_E=0\Longleftrightarrow f=0\ \  a.e.$;
   \item[$(B1)$] if $f\in E$, then $\||f|\|_E=\|f\|_E$;
   \item[$(B2)$] if $0\leq g\leq f$, then $\|g\|_E\leq \|f\|_E$;
   \item[$(B3)$] if $0\leq f_n\uparrow f$, then $\|f_n\|_E\uparrow \|f\|_E$;
   \item[$(B4)$] if $A\subset \rn$ is bounded, then $\chi_A\in E$.
 \end{description}
\end{definition}
For a definition of classical Banach function space (CBFS) used in \cite{CaetanoGogatishviliOpic2016PRSESA,GorkaRafeiro2016NA},
one can see the books Bennett-Sharpley \cite{BennettSharpley1988} and Edmunds-Evans \cite{EdmundsEvans2004}, where
$(B4)$ is replaced by
 \begin{description}
\item[$(B4^*)$] if $A\subset \rn$ with $m(A)<\fy$, then $\chi_A\in E$,
 \end{description}
and the following condition is needed
 \begin{description}
\item[$(B5)$] if $A\subset \rn$ and $m(A)<\fy$, then there exists a constant $C(A)$ such that
\be
\int_A|f|dx\leq C(A)\|f\|_E.
\ee
 \end{description}

\begin{remark}\label{remark, 1}
Note that Q-BFS in Definition \ref{df, B} is coincide with \cite[Definition 3.1]{CaetanoGogatishviliOpic2016PRSESA}.
However, in \cite[Remark 3.4]{CaetanoGogatishviliOpic2016PRSESA}, one can see that the terminology "Banach function space"  in \cite{CaetanoGogatishviliOpic2016PRSESA} is used to denote CBFS (not BFS).
We emphasize again that BFS in Definition \ref{df, B} does not need to meet condition $(B_5)$.
We recall that the Q-BFS in Definition \ref{df, B}
is a (Quasi-)Banach space (see \cite[Lemma 3.6]{CaetanoGogatishviliOpic2016PRSESA}).
\end{remark}

Let $L^0(m)$ denote the class of functions in $L(m)$ that are finite almost everywhere,
with the topology of convergence in measure on sets of finite measure.
We recall that Q-BFS is continuous embedded into $L^0(m)$, see \cite[Lemma 3.3]{CaetanoGogatishviliOpic2016PRSESA}.
\begin{lemma}[\cite{CaetanoGogatishviliOpic2016PRSESA}]\label{lemma, L0}
  Let $E$ be a Q-BFS. Then $E$ is continuous embedded into $L^0(m)$.
  In particular, if $f_k$ tends to $f$ in $E$, then $f_k$ tends to $f$ in measure on sets of finite measure, and
  hence some sequence convergence pointwise to $f$ a.e.
\end{lemma}

Moreover, we recall following two definitions.
\begin{definition}[absolutely continuous quasi-norm]
Let $E$ be a Q-BFS.
  A function $f$ in $E$ is said to have absolutely continuous quasi-norm in $E$
  if $\|f\chi_{A_n}\|\rightarrow 0$ as $E_n\rightarrow \emptyset$.
  The set of all functions in $E$ with absolutely continuous quasi-norm is denoted by $E_a$.
  If $E=E_a$, then the space $E$ is said to have absolutely continuous quasi-norm.
\end{definition}
We point out that the dominated convergence theorem holds in Q-BFS with absolutely continuous quasi-norm,
see \cite[Proposition 3.9]{CaetanoGogatishviliOpic2016PRSESA}.
\begin{definition}[uniformly absolutely continuous quasi-norm]
Let $K$ be a Q-BFS, and $K\subset E_a$. Then $K$ is said to have uniformly absolutely continuous quasi-norm $(K\subset UAC(E))$
if for every sequence $\{A_k\}_{k=1}^{\fy}$ with $A_k\rightarrow \emptyset$, $\|f\chi_{A_k}\|_E\rightarrow 0$ holds
uniformly for all $f\in K$.

\end{definition}

In harmonic analysis, a weight is a nonnegative locally integrable function on $\rn$ that takes values in
 $(0,\fy)$ almost everywhere (see \cite{Grafakos2014}).
 For a weight function $\om$, the $\lpw$ function space with $p\in (0,\fy)$ is defined by
 \be
 \lpw:=\bigg\{f\in L^0(m): \|f\|_{\lpw}:=\left(\int_{\rn}|f(x)|^p\om(x)dx\right)^{1/p}<\fy\bigg\}.
 \ee
 In order to find out whether the relative compactness criteria in \cite{CaetanoGogatishviliOpic2016PRSESA,GorkaRafeiro2016NA} can be used in the case of weighted function spaces, we first point out that in the proofs of
 \cite[Theorem 4.3]{CaetanoGogatishviliOpic2016PRSESA} and
 \cite[Theorem 3.1]{GorkaRafeiro2016NA}, the authors actually only use $(B5)$ with bounded set $A$.
 In the case of weighted function spaces $E=\lpw$ with $1\leq p<\fy$,
 the $(B5)$ condition with bounded set $A$ is just
 \ben\label{s2, 2}
 \int_A|f|dx\leq C(A)\left(\int_{A}|f(x)|^p\om(x)dx\right)^{1/p}\ \  \text{for any bounded set}\ A,
 \een
 which is equivalent to
 \ben\label{s2, 1}
 \int_A\om(x)^{1-p'}dx<\fy\ \  \text{for any bounded set}\ A,
 \een
 where the right term should be interpreted as $\|\om^{-1}\chi_A\|_{L^{\fy}}$ for $p=1$.
It is well known that \eqref{s2, 1} can be deduced by the so-called $A_p$ condition or be as an independent assumption as in \cite[Lemma 4.1]{XueYabutaYan2018apa}.
We recall the definition of $A_p$ as follows.
\begin{definition}[\cite{Grafakos2014}]\label{d-Ap weight}
For $1<p<\infty$, the Muckenhoupt class $A_p$ is the set of locally integrable weights $\om$ such that
\be
[\om]_{A_p}^{1/p}:=\sup_{Q}\left(\frac{1}{|Q|}\int_{Q}\om(x)dx\right)^{1/p}\left(\frac{1}{|Q|}\int_{Q}\om(x)^{1-p'}dx\right)^{1/p'}<\infty,
\ee
where $1/p+1/p'=1$, $Q$ denotes the cubes in $\rn$.
\end{definition}
\begin{definition}[\cite{Grafakos2014}]
  A weight function $\om$ is called an $A_1$ weight if
  \be
  \frac{1}{|Q|}\int_{Q}\om(x)dx\leq [\om]_{A_1}\text{ess.inf}_{y\in Q}\om(y),
  \ee
  where $Q$ denotes the cubes in $\rn$.
\end{definition}
Note that the condition \eqref{s2, 1} holds for $A_p$ weight.
Furthermore, in \cite{GorkaRafeiro2016NA} if we choose the Banach function space $E=\lpw$ containing functions belong to $L^0(m)$,
and observe that in \cite{GorkaRafeiro2016NA} the authors actually only use $(B5)$ with bounded set $A$,
then the following result is a direct conclusion of \cite[Theorem 3.1]{GorkaRafeiro2016NA}.
\begin{proposition}\label{prop}
  Let $1\leq p<\fy$, and
  let $\om$ be a weight satisfies \eqref{s2, 1}. If the subset $\calF$ of $\lpw$ satisfies the following three conditions:
  \begin{description}
  \item [$(a)$] $\calF$ is bounded, i.e.,
  \be
  \sup_{f\in \calF}\|f\|_{\lpw}\lesssim 1;
  \ee
  \item [$(b)$] $\calF$ uniformly vanishes at infinity, that is,
  \be
  \lim_{N\rightarrow \infty}\sup_{f\in \calF}\|f\chi_{B^c(0,N)}\|_{\lpw}= 0;
  \ee
  \item [$(c^*)$] $\calF$ is equicontinuous, that is,
  \be
  \lim_{r\rightarrow 0}\sup_{f\in \calF}\left\|\frac{1}{B(\cdot,r)}\int_{B(\cdot,r)}f(y)dy-f\right\|_{\lpw}=0,
  \ee
  then $\calF$ is a totally bounded subset of $\lpw$.
\end{description}
\end{proposition}

From \eqref{intro, 1} or \cite[Theorem 4.3]{CaetanoGogatishviliOpic2016PRSESA},
the condition $(c^*)$ in Proposition \ref{prop} can be replaced by the following stronger one:
\ben
(c)\ \ \ \ \ \ \ \ \ \ \ \ \ \ \lim_{r\rightarrow 0}\sup_{f\in \calF}\sup_{y\in B(0,r)}\|\tau_yf-f\|_{\lpw}=0.
\een
However, for more general weights, $(B5)$ is too strong to apply.
In order to explain this more precisely, we give a counterexample here.
Let $E=\lpw$ with $\om(x)=|x|^{n(p-1)+1}$, $f_N=\chi_{B(0,1/N)}$ and $A=B(0,1)$.
A direct calculation yields that
\be
\sup_{f\in X}\frac{\int_A|f|dx}{\|f\|_E}\geq \frac{\int_{\rn}f_Ndx}{\|f_N\|_{\lpw}}\sim \frac{N^{-n}}{N^{-n-1/p}}=N^{1/p}\rightarrow \fy,
\ \ \ \ \text{as}\  N\rightarrow \fy.
\ee
This breaks the condition \eqref{s2, 2}.
This counterexample shows the invalidation of conclusions in \cite{CaetanoGogatishviliOpic2016PRSESA,GorkaRafeiro2016NA}, in which
$(B_5)$ condition can not be removed.

\begin{remark}\label{remark, 2}
Note that, for any weight $\om$, $L^p_{\om}(\rn) (1\leq p<\fy)$ even if not satisfy \eqref{s2, 1}, is still a BFS.
So the additional assumption on weight function can be completely eliminated in our framework.
\end{remark}

\section{relatively compactness criteria}\label{s3}
In this section, we give and prove our main theorems, including the relatively compactness criteria
on BFS and Q-BFS, and including an important application on weighted Lebesgue space $\lpw$.
\subsection{on Banach function space}
In this subsection, we establish the relatively compactness criteria in the framework of Banach function space.
As mentioned above, we drop the assumption $(B5)$, providing a more general framework fitting weighted function spaces with any weight.
Meanwhile, in our following theorem, the equicontinuous condition is chosen to be the ``$(c)$ type'' as in Theorem A. Although the "$(c^*)$ type" as in
Proposition \ref{prop}
is weaker, however, the relatively compactness criteria of ``$(c^*)$ type'' relies heavily on the condition $(B5)$, one can see
the proof of \cite[Theorem 3.1]{GorkaRafeiro2016NA} for more details.
On the other hand, the ``$(c)$ type'' condition is more applicable in the field of harmonic analysis, see \cite{Uchiyama1978TMJ2,ChenDingWang2012CJMa, BenyiDamianMoenTorres2015MMJ, ChaffeeChenHanTorresWard2018PAMS, ClopCruz2013AASFM, WuYang2018PAMS}.
Therefore, in the framework of BFS, ``$(c)$ type'' condition is reasonable and has strong applicability.

\begin{theorem}\label{thm, cc}
  Let $E$ be a BFS with absolutely continuous norm. If the family $\calF\subset E$ satisfies the following conditions:
\bn[(a)]
  \item $\calF$ is bounded, i.e.,
  \be
  \sup_{f\in \calF}\|f\|_{E}\lesssim 1;
  \ee
  \item $\calF$ uniformly vanishes at infinity, that is,
  \be
  \lim_{N\rightarrow \infty}\sup_{f\in \calF}\|f\chi_{B^c(0,N)}\|_{E}= 0;
  \ee
  \item $\calF$ is equicontinuous, that is,
  \be
  \lim_{r\rightarrow 0}\sup_{f\in \calF}\sup_{y\in B(0,r)}\|\tau_yf-f\|_E=0,
  \ee
  \en
  then the family $\calF$ is a totally bounded subset of $E$.
\begin{proof}
  From condition (c), there exists a sufficiently small $\d\in (0,1)$ such that
\be
  \tau_yf\in E\ \text{for all}\ y\in [-\d,\d]^n.
\ee
We first claim that $\|\tau_yf-f\|_E\chi_{[-\d,\d]^n}(y)$ and $\|\tau_yf\|_E\chi_{[-\d,\d]^n}(y)$
are measurable.
This is clear for some specific cases such as $E=L^{p}_{\om}$ or $E$ is rearrangement invariant.
However, it is not obvious in our general case.
Here, we only give the detailed proof for the measurability of $\|\tau_yf-f\|_E\chi_{[-\d,\d]^n}(y)$, since
the proof for $\|\tau_yf\|_E\chi_{[-\d,\d]^n}(y)$ is similar and easier.

First, we choose a sequence $\{f_k\}_{k=1}^{\fy}$ of compact supported simple functions such that
\be
|f_k|\leq |f|\ \ \text{and}\ \ \lim_{k\rightarrow \fy}f_k=f\ a.e.
\ee
From this, for fixed $y\in [-\d,\d]^n$,
\be
|\tau_yf_k-f_k|\leq |\tau_yf|+|f|\in E\ \ \text{and}\ \ \lim_{k\rightarrow \fy}(\tau_yf_k-f_k)=\tau_yf-f\  a.e.
\ee
Then, the dominated convergence theorem yields that $\tau_yf_k-f_k\rightarrow \tau_yf-f$ in $E$, which implies that
\be
\lim_{k\rightarrow \fy}\|\tau_yf_k-f_k\|_E=\|\tau_yf-f\|_E.
\ee
Note that the above pointwise convergence is valid for all $y\in [-\d,\d]^n$.
Thus, in order to verify the measurability of $\|\tau_yf-f\|_E\chi_{[-\d,\d]^n}(y)$, we only need
to consider the case that $f$ is a compact supported simple function,
in which we can actually verify that
$\|\tau_yf-f\|_E$ is measurable.

Let us continue the process of regularising $f$. Without loss of generality, we assume that $f$ is supported on $B(0,N)$, with bounded $N$.
Take a nonnegative smooth function $\varphi$ with compact support on $B(0,1)$ such that
\be
\int_{\rn}\varphi(x)dx=1.
\ee
Denote by $\varphi_t(x):=t^{-n}\varphi(x/t)$, $t\in (0,1)$. Set
\be
f_t(x)=f\ast \varphi_t(x)=\int_{B(0,1)}f(x-ty)\varphi(y)dy.
\ee
Then $f_t$ and $\tau_yf_t$ have uniform bound and uniform support for all $t\in (0,1)$. Since
\be
\text{supp}f_t\subset \text{supp}f+\text{supp}\va_t\subset B(0,N+1),\  \text{supp}\tau_yf_t\subset B(0,N+1+|y|)
\ee
and
\be
\|\tau_yf_t\|_{L^{\fy}(\rn)}=\|f_t\|_{L^{\fy}(\rn)}\leq \|f\|_{L^{\fy}(\rn)}\|\va\|_{L^1(\rn)}=\|f\|_{L^{\fy}(\rn)}\leq N.
\ee
Thus, for fixed $y$ and all $t\in (0,1)$ we have
\ben\label{1}
|\tau_yf_t-f_t|\leq 2N\chi_{B(0,N+1+|y|)}.
\een
Since $f\in L^1(\rn)$, we obtain
\be
\lim_{t\rightarrow 0}f_t(x)=f(x)\ \  a.e.
\ee
In fact, the above convergence is valid for every $x$ in the Lebesgue set of $f$ (see \cite[Theorem 2.1]{SteinShakarchi2005}).
From this, for fixed $y$ we have
\ben\label{2}
\lim_{t\rightarrow 0}(\tau_yf_t(x)-f_t(x))=\tau_yf(x)-f(x)\ \  a.e.
\een
Using \eqref{1}, \eqref{2} and the dominated convergence theorem, we have $\tau_yf_t(x)-f_t(x)\rightarrow \tau_yf(x)-f(x)$ in $E$.
In particular, for fixed $y$
\be
\lim_{t\rightarrow 0}\|\tau_yf_t-f_t\|_E=\|\tau_yf-f\|_E.
\ee
Again, the above pointwise convergence is valid for all $y\in \rn$. On the other hand, one can easily check that $f_t\in C_c^{\fy}(\rn)$ for all $t\in (0,1)$.
Thus, in order to verify the measurability of $\|\tau_yf-f\|_E$ for a compact supported simple function $f$,
we only need
to consider the case that $f$ is a smooth function with compact support. However, this is obvious, since
for fixed $x$ and $y\in B(x,1)$ we have
$|\tau_yf-\tau_xf|\leq |y-x|\|\nabla f\|_{L^{\fy}}\chi_{B(0,M)}$ for some $M>0$ depends on $x$ and the support of $f$.
Thus,
\be
\begin{split}
\big|\|\tau_yf-f\|_E-\|\tau_xf-f\|_E\big|
\leq &
\|\tau_yf-\tau_xf\|_E
\\
\leq &
|x-y|\|\nabla f\|_{L^{\fy}}\|\chi_{B(0,M)}\|_E,
\end{split}
\ee
where the last term tends to zero as $y$ tends to $x$.
Thus, the map $y\rightarrow \|\tau_yf-f\|_E$ is continuous, which implies the measurability of $\|\tau_yf-f\|_E$.
This completes the proof of measurability.

Next, we proceed to the totally boundedness of $\calF$.
We only need to find the finite $\ep$-net of $\calF$ for each fixed $\ep$.
  Denote by $R_i:= [-2^{i},2^i]^n$ for $i\in \mathbb{Z}$.
  By condition (b), there exists a sufficiently large positive integer $m$ such that
  \be
  \|f-f\chi_{R_m}\|_{E}<\ep/3.
  \ee
  Thus, we only need to verify that the family of functions $\{f\chi_{R_m}\}_{f\in \calF}$
  has a finite $\frac{2\ep}{3}$-net.
  By condition (c), we choose an integer $i_{\ep}$ such that $2^{i_{\ep}}<\d$ and
  \be
  \|\tau_yf-f\|_E<2^{-n}\ep/3,\ \ \ y\in R_{i_{\ep}}.
  \ee
  For $x\in R_{m}$, $Q_x$ means the dyadic cube $\prod_{j=1}^n[m_j2^{i_{\ep}},(m_j+1)2^{i_{\ep}})$ of side length $2^{i_{\ep}}$ that contains $x$, for some integers $m_j$. Obviously, any two dyadic cubes $Q_x$ and $Q_y$ either are disjoint or coincide.
  Define
  \be
  \Phi(f\chi_{R_m})(x)=
  \begin{cases}
    f_{Q_x}:=\frac{1}{|Q_x|}\int_{Q_x}f(y)dy,\ &x\in R_m,
    \\
    0,\ \ \ &\text{otherwise}.
  \end{cases}
  \ee
  We claim that the map $\Phi$ is well-defined by
  \be
  \int_{Q_x}|f(y)|dy<\fy,\ \ \ \text{for}\ \ x\in R_m.
  \ee
  It follows by the Minkowski-type inequality that
  \be
  \begin{split}
  \left\|\int_{R_{i_{\ep}}}|f(\cdot-y)|dy\right\|_E
  \leq &
\int_{R_{i_{\ep}}}\left\|f(\cdot-y)\right\|_Edy
  \\
  \leq &
  \int_{R_{i_{\ep}}}\|\tau_yf-f\|_Edy
  +
  \int_{R_{i_{\ep}}}\|f\|_Edy
  \\
  \leq &
  (2\d)^n \sup_{y\in [-\d,\d]^n}\|\tau_yf-f\|_E<\fy+|R_{i_{\ep}}|\|f\|_E<\fy.
  \end{split}
  \ee
  For any fixed $Q_x$, there exists a point $x_0\in Q_x$ such that
  \be
  \int_{R_{i_{\ep}}}|f(x_0-y)|dy<\fy.
  \ee
Observing that $Q_x\subset x_0-R_{i_{\ep}}$, we further have

\be
\int_{Q_x}|f(y)|dy
\leq
\int_{x_0-R_{i_{\ep}}}|f(y)|dy
=
\int_{R_{i_{\ep}}}|f(x_0-y)|dy<\fy.
\ee

Next, we turn to the estimate of $\|f\chi_{R_m}-\Phi(f\chi_{R_m})\|_{E}$.
A direct calculation yields that
\be
\begin{split}
|(f-f_{Q_x})\chi_{Q_x}|
= &
\left|\frac{1}{|Q_x|}\int_{Q_x}\left(f(x)-f(z)\right)dz\chi_{Q_x}\right|
\\
\leq &
\frac{1}{|Q_x|}\int_{Q_x}|f(x)-f(z)|dz\chi_{Q_x}
\\
\leq &
\frac{1}{|Q_x|}\int_{R_{i_{\ep}}}|f(x)-f(x-y)|dy\chi_{Q_x}.
\end{split}
\ee
It follows that
\be
\begin{split}
  \left|\sum_{Q_x\subset R_m}(f-f_{Q_x})\chi_{Q_x}\right|
  \leq &
\sum_{Q_x\subset R_m}\left|(f-f_{Q_x})\chi_{Q_x}\right|
\\
\leq &
\sum_{Q_x\subset R_m}\frac{1}{|Q_x|}\int_{R_{i_{\ep}}}|f(x)-f(x-y)|dy\chi_{Q_x}
\\
= &
2^{-n i_{\ep}}\int_{R_{i_{\ep}}}|f(x)-f(x-y)|dy\chi_{R_m}(x)
\end{split}
\ee

  Hence,
  \be
  \begin{split}
    \|f\chi_{R_m}-\Phi(f\chi_{R_m})\|_E
    = &
    \left\|\sum_{Q_x\subset R_m}f\chi_{Q_x}-\sum_{Q_x\subset R_m}f_{Q_x}\chi_{Q_x}\right\|_E
    \\
    = &
    \left\|\sum_{Q_x\subset R_m}(f-f_{Q_x})\chi_{Q_x}\right\|_E
    \\
    \leq &
    2^{-n i_{\ep}}\left\|\int_{R_{i_{\ep}}}|f(x)-f(x-y)|dy\chi_{R_m}(x)\right\|_E
  \end{split}
  \ee
  Applying the Minkowski-type inequality, we have
  \be
  \begin{split}
  & 2^{-n i_{\ep}}\left\|\int_{R_{i_{\ep}}}|f(x)-f(x-y)|dy\chi_{R_m}(x)\right\|_E
  \\
  \leq &
  2^{-n i_{\ep}}\int_{R_{i_{\ep}}}\|\tau_yf-f\|_Edy\leq 2^n\sup_{y\in R_{i_{\ep}}}\|\tau_yf-f\|_E<\ep/3.
  \end{split}
  \ee
  The above two estimates imply that
  \be
  \|f\chi_{R_m}-\Phi(f\chi_{R_m})\|_E<\ep/3.
  \ee
  From this, to get our final conclusion, we only need to verify that the family of functions $\{\Phi(f\chi_{R_m})\}_{f\in \calF}$
  has a finite $\frac{\ep}{3}$-net. This is true since this family is a bounded  subset of a finite dimensional Banach space.
  Let us check the boudedness by
  \be
  \begin{split}
  \|\Phi(f\chi_{R_m})\|_E
  \leq &
  \|f\chi_{R_m}-\Phi(f\chi_{R_m})\|_E+\|f\chi_{R_m}\|_E
  \\
  \leq &
  \ep/3+\|f\|_E\lesssim 1.
  \end{split}
  \ee
  Now, we have completed this proof.
\end{proof}
\end{theorem}
\begin{remark}
Here, the proof is finished by a finite dimension argument.
By adapting the arguments in \cite{Hanche-OlsenHolden2010EM},
  the authors in \cite{ClopCruz2013AASFM} also used a finite dimension argument
  to prove the relatively compactness criteria on $\lpw$ with $\om\in A_{p}(\rn)$.
  Unfortunately, the method in \cite{ClopCruz2013AASFM, Hanche-OlsenHolden2010EM} is invalid here, since
  it depends heavily on the $A_{p}(\rn)$ condition and the special structure of $\lpw$.
  By contrast, our new method is applicable for more general
  Banach function spaces including weighted function spaces with any weight.
\end{remark}

\subsection{on weighted function space}
Obviously, Theorem \ref{thm, cc} can be applied to $\lpw$ spaces with any weights.
Here, we would like to show a more general case in which the ``weight function $v$'' can even disappear on a positive measurable set.
\begin{theorem}\label{thm, ccbw}
  Let $1\leq p<\fy$, $v\in L^0(m)$ be a nonnegative function. Define
  \be
  \|f\|_{L^p_v(\rn)}:=\left(\int_{\rn}|f(x)|^pv(x)dx\right)^{1/p}.
  \ee
  If the family $\calF\subset L_v^p(\rn)$ satisfies the following conditions:
\bn[(a)]
  \item $\calF$ is bounded, i.e.,
  \be
  \sup_{f\in \calF}\|f\|_{L^p_v(\rn)}\lesssim 1;
  \ee
  \item $\calF$ uniformly vanishes at infinity, that is,
  \be
  \lim_{N\rightarrow \infty}\sup_{f\in \calF}\|f\chi_{B^c(0,N)}\|_{L^p_v(\rn)}= 0;
  \ee
  \item $\calF$ is equicontinuous, that is,
  \be
  \lim_{r\rightarrow 0}\sup_{f\in \calF}\sup_{y\in B(0,r)}\|\tau_yf-f\|_{L^p_v(\rn)}=0,
  \ee
  \en
  then the family $\calF$ is a totally bounded subset of $L^p_v(\rn)$.
\end{theorem}
\begin{proof}
  As in the proof of Theorem \ref{thm, cc},
   we only need to verify that the family of functions $\{f\chi_{R_m}\}_{f\in \calF}$
  has a finite $\frac{2\ep}{3}$-net.
  Take $i_{\ep}$ as in the proof of Theorem \ref{thm, cc}.
  For $x\in R_{m}$, $Q_x$ means the dyadic cube of side length $2^{i_{\ep}}$ that contains $x$.
  Define
  \be
  \Phi(f\chi_{R_m})=
  \begin{cases}
    f_{Q_x},\ &x\in R_m,\ \|\chi_{Q_x}\|_{L^p_v(\rn)}\neq 0,
    \\
    0,\ \ \ &\text{otherwise}.
  \end{cases}
  \ee
  We claim that the map $\Phi$ is well-defined by
  \be
  \int_{Q_x}|f(y)|dy<\fy,\ \ \ \text{for}\ \ x\in R_m,\ \|\chi_{Q_x}\|_{L^p_v(\rn)}\neq 0.
  \ee
  In fact, by the same estimate as in the proof of Theorem \ref{thm, cc} we have
  \be
  \begin{split}
  \left\|\int_{R_{i_{\ep}}}|f(\cdot-y)|dy\right\|_{L^p_v(\rn)}<\fy.
  \end{split}
  \ee
  From this, we have
  \be
  \left|\int_{R_{i_{\ep}}}|f(x-y)|dy\right|^pv(x)<\fy\ \ \ \text{a.e.}\ x\in \rn.
  \ee
  If $\|\chi_{Q_x}\|_{L^p_v(\rn)}\neq 0$, then $v(x)\neq 0$ on a positive measurable subset of $Q_x$.
  Then, there exists a point $x_0\in Q_x$ such that
  \be
  \int_{R_{i_{\ep}}}|f(x_0-y)|dy<\fy.
  \ee
Observing that $Q_x\subset x_0-R_{i_{\ep}}$, the reasonable definition of $\Phi$ is assured by
\be
\int_{Q_x}|f(y)|dy
\leq
\int_{x_0-R_{i_{\ep}}}|f(y)|dy
=
\int_{R_{i_{\ep}}}|f(x_0-y)|dy<\fy.
\ee
The remaining part of this proof is the same as the proof of Theorem \ref{thm, cc}.
\end{proof}

\subsection{on Quasi-Banach space}
In the field of harmonic analysis, one may deal with the boundedness and compactness of certain operators
on Quasi-Banach spaces. An important situation occurs in the multilinear setting, where the multilinear operator can be
bounded or compact into a Quasi-Banach space such as $\lpw$ with $p\in (0,1)$, see \cite{GrafakosTorres2002AM, ChaffeeChenHanTorresWard2018PAMS}.
Thus, it is very important to establish a corresponding relatively compactness criteria on Quasi-Banach spaces.
In fact, the relatively compactness criteria in Theorem \ref{thm, cc}
can be also valid for the power Q-BFS, providing
an important application on $\lpw$ with $p<1$.
\begin{definition}[Power Q-BFS]\label{def, pqbfs}
  A Q-BFS $E$ is called a power Q-BFS if there exists a constant $b\in (0,1]$ such that
  $\|\cdot\|_{E^b}:=\||\cdot|^{1/b}\|_E^{b}$ is a norm and $E^b\subset L^0(m)$ is a BFS.
\end{definition}
Denote by $\|\cdot\|_{E^b}:=\||\cdot|^{1/b}\|_E^{b}$.
We point out that when $p\in (0,1)$, $L^p_{\om}$ is a power Q-BFS
with $b=p$.
Recall that our definition is weaker than \cite[Definition 4.5 and Theorem 4.9]{CaetanoGogatishviliOpic2016PRSESA},
since $E^b$ is assumed to be a CBFS there, see also Remark \ref{remark, 1}.
Because of this, the conclusion of \cite[Theorem 4.9]{CaetanoGogatishviliOpic2016PRSESA}
can not be used for $L^p_{\om}$ with any weight, see also Remark \ref{remark, 2}.
However, the method used in the proof of \cite[Theorem 4.9]{CaetanoGogatishviliOpic2016PRSESA}
inspired our proof of relatively compactness criteria for power Q-BFS.
Before stating our theorem, we first recall a characterization of relatively compact set in Q-BFS.
We would like to note that the condition $(B4)$ in our Definition \ref{df, B} is enough to prove the following lemma.
\begin{lemma}[see Theorem 3.17 in \cite{CaetanoGogatishviliOpic2016PRSESA}]\label{lemma, cha of precompact}
Let $E$ be a Q-BFS and $K\subset E_a$. Then $K$ is relatively compact in $E$ if and only if
it is locally relatively compact in measure and $K\subset UAC(E)$.
\end{lemma}

Now, we state our theorem as follows. The proof follows by \cite[Theorem 4.9 (b)]{CaetanoGogatishviliOpic2016PRSESA}
with some slight modifications fitting our case.

\begin{theorem}\label{thm, ccq}
The conclusion of Theorem \ref{thm, cc} is also valid
if $E$ is a power Q-BFS with absolutely continuous norm.
\end{theorem}

\begin{proof}
Denote by
\be
Re_{\pm}\calF=\{(Ref)_{\pm}: f\in \calF\},\ Im_{\pm}\calF=\{(Imf)_{\pm}: f\in \calF\}.
\ee
Observe that
\bn
\item If $\calF$ satisfies $(a)-(c)$ in Theorem \ref{thm, cc}, so do $Re_{\pm}\calF$ and $Im_{\pm}\calF$;
\item If $Re_{\pm}\calF$ and $Im_{\pm}\calF$ are all relatively compact in $E$, so does $\calF$.
\en
So, we only need to consider the case that $\calF$ only consists of nonnegative functions.
Choose a constant $b$ as in Definition \ref{def, pqbfs}.
Note that $E^b$ is a power Q-BFS with absolutely continuous norm.
Denote by $\calF^b: =\{f^{b}: f\in \calF\}$.
We claim that $\calF^b$ is relatively compact in $E^b$.
In fact, $\calF^b$ satisfies all the conditions $(a)-(c)$ in Theorem \ref{thm, cc}
with norm $\|\cdot\|_{E^b}$. We only check the condition $(c)$ by
\be
\|\tau_yf^{b}-f^{b}\|_{E^b}\leq \||\tau_yf-f|^{b}\|_{E^b}=\|\tau_yf-f\|_E^{b}.
\ee

Next, we will see that the relatively compactness of $\calF^b$ in $E^b$ implies the relatively compactness of $\calF$ in $E$.
Take any sequence $\{f_k\}_{k=1}^{\fy}$ of $\calF$. By the relatively compactness of $\calF^b$ in $E^b$, there exists a subsequence of
$\{f^b_k\}_{k=1}^{\fy}$, still denoted by $\{f^b_k\}_{k=1}^{\fy}$, which tends to $f^b$ in $E^b$. Using Lemma \ref{lemma, L0},
$f^b_k\rightarrow f^b$ locally in measure.
Next, by choosing the diagonal subsequence from each subset with finite measure, we further find a
subsequence of $\{f^b_k\}_{k=1}^{\fy}$, still denoted by $\{f^b_k\}_{k=1}^{\fy}$,  pointwise  tends to $f^b$ a.e.
From this, $f_k\rightarrow f$ pointwise a.e., which further implies that $f_k\rightarrow f$ locally in measure.
Note that $f\in E$, we actually verify that $\calF$ is locally relatively compact in measure.

On the other hand, since $\calF^b$ is relatively compact in $E^b$, then $\calF^b\subset UAC(E^b)$ by Lemma \ref{lemma, cha of precompact}.
One can easily verify that $\calF^b\subset UAC(E^b)$ implies $\calF\subset UAC(E)$.

Now, we have verified that $\calF\subset UAC(E)$, and $\calF$ is locally relatively compact in measure.
The relatively compactness of $\calF$ follows by
 Lemma \ref{lemma, cha of precompact}.
\end{proof}

\begin{remark}
The proof of Theorem \ref{thm, cc}, based on the Minkowski-type inequality, is invalid for the Quasi-Banach case here.
In order to deal with the Quasi-Banach case, we use Lemma \ref{lemma, cha of precompact} to reduce the proof to the Banach case.
We remark that the original idea of reduction to the Banach case of $L^p_{\om}$ should be due to Tsuji \cite{Tsuji1951KMSR}.
\end{remark}
As a useful conclusion, we show the final version of relatively compactness criteria for $\lpw$ as follows.
\begin{corollary}[Relatively compactness criteria for $\lpw$ with any weight]\label{lemma, general FK theorem}
  Let $p\in (0,\infty)$, $\om$ be a weight.
  A subset $\calF$ of $\lpw$ is relatively compactness (or totally bounded) if
  the following statements are valid:
  \bn[(a)]
  \item $\calF$ is bounded, i.e., $\sup\limits_{f\in \calF}\|f\|_{L^p(\om)}\lesssim 1$;
  \item $\calF$ uniformly vanishes at infinity, that is,
  \be
  \lim_{N\rightarrow \infty}\sup_{f\in \calF}\|f\chi_{B^c(0,N)}\|_{\lpw}= 0;
  \ee
  \item $\calF$ is equicontinuous, that is,
  \be
  \lim_{r\rightarrow 0}\sup_{f\in \calF}\sup_{y\in B(0,r)}\|\tau_yf-f\|_{\lpw}=0.
  \ee
  \en
  \end{corollary}
  \begin{proof}
    Take $b=\min\{1,p\}$,
    this conclusion follows by Theorem \ref{thm, ccq}.
  \end{proof}

\appendix
\section{}
In the case of CBFS, the Minkowski-type inequality follows by the application of associate space, see \cite[Lemma 4.2]{CaetanoGogatishviliOpic2016PRSESA}.
However, the absence of $(B_5)$ in BFS leads to the useless of previous method.
Hence, we would like to give a proof for the Minkowski-type inequality for BFS with absolutely continuous norm.
The key point is to linearise the norm. For this, we recall a result in \cite[Theorem 4.1]{BennettSharpley1988}.
\begin{lemma}\label{lemma, linearization}
  Let $E$ be a BFS with absolutely continuous norm. For every $T\in E^{\ast}$, there exists a measurable function $g$ such that
  \be
  T(f)=\int_{\rn}f(x)g(x)dx\ \ \text{for all}\ f\in E
  \ee
  with
  \be
  \sup_{\|f\|_E=1}\int_{\rn}|f(x)g(x)|dx \leq 2\|T\|_{E^*}.
  \ee
\end{lemma}
\begin{proposition}[Minkowski-type inequality]
Let $E$ be a BFS with absolutely continuous norm. Suppose that $f$ is a nonnegative measurable function on $\rn\times\rn$
with $\int_{\rn}\|f(\cdot,y)\|_Edy<\fy$, then
\be
\left\|\int_{\rn}f(\cdot,y)dy\right\|_E\leq \int_{\rn}\|f(\cdot,y)\|_Edy.
\ee
\begin{proof}
  Denote by $h(x):=\int_{\rn}f(x,y)dy$, we first claim that $h\in E$.
  Take a sequence $\{h_k\}_{k=1}^{\fy}$
  of simple functions such that $0\leq h_k\uparrow h$ a.e. and $h_k\in E$.
  For fixed $k$, by Hahn-Banach theorem there exists a bounded linear functional $T_k$ such that
  \ben\label{ap, 1}
  T_kh_k=\|h_k\|_E\ \ \ \ \text{and}\ \ \ \ \ \|T_k\|_{E^*}=1.
  \een
  Then, Lemma \ref{lemma, linearization} yields that there exists a function $g_k$ such that
  \ben\label{ap, 2}
  T_kf=\int_{\rn}f(x)g_k(x)dx
  \een
  and
  \ben\label{ap, 3}
  \int_{\rn}|f(x)g_k(x)|dx \leq 2\|T_k\|_{E^*}\|f\|_E=2\|f\|_E \ \ \ \forall f\in E.
  \een

  Using \eqref{ap, 1} and \eqref{ap, 2}, we obtain
  \be
  \begin{split}
    \|h_k\|_E=T_kh_k
    = &\int_{\rn}h_k(x)g_k(x)dx
    \leq
    \int_{\rn}h_k(x)|g_k(x)|dx
    \leq
    \int_{\rn}h(x)|g_k(x)|dx.
  \end{split}
  \ee
  By \eqref{ap, 3} and the fact $h(x)=\int_{\rn}f(x,y)dy$, we further have
  \be
  \begin{split}
    \|h_k\|_E
    \leq &
    \int_{\rn}\int_{\rn}f(x,y)dy|g_k(x)|dx
    \\
    = &
    \int_{\rn}\int_{\rn}f(x,y)|g_k(x)|dx dy
    \leq
    \int_{\rn}2\|f(\cdot,y)\|_E dy<\fy.
  \end{split}
  \ee
  Letting $k\rightarrow \fy$, we deduce that $h\in E$ by $(B_3)$ of Definition \ref{df, B}.

  Now, since $h\in E$, replacing $h_k$ by $h$, the similar proof as above yields that
  there exists a measurable function $g$ such that
  \be
  \begin{split}
    \|h\|_E
    =
    \int_{\rn}\int_{\rn}f(x,y)dy g(x)dx
    =
    \int_{\rn}\int_{\rn}f(x,y)g(x)dx dy
    \leq
    \int_{\rn}\|f(\cdot,y)\|_E dy,
  \end{split}
  \ee
  where we use Fubini's theorem in the last equality.  Now, this proof has been completed.
\end{proof}

\end{proposition}

\subsection*{Acknowledgements}
The authors sincerely appreciate the anonymous referees for checking this paper very carefully
and giving very useful comments, which greatly improved this article.



\end{document}